\documentclass{article}

\usepackage[english]{babel}

\usepackage[a4paper,top=2cm,bottom=2cm,left=3cm,right=3cm,marginparwidth=1.75cm]{geometry}

\usepackage{amsmath, amssymb}
\usepackage{graphicx}
\usepackage[colorlinks=true, allcolors=blue]{hyperref}
\usepackage{mathtools}
\usepackage{amsfonts}
\usepackage{amsthm}
\usepackage{tikz-cd} 

\newtheorem{thm}{Theorem}[section]

\newtheorem{lem}[thm]{Lemma}
\newtheorem{prop}[thm]{Proposition}
\newtheorem{rem}[thm]{Remark}
\newtheorem{exam}[thm]{Example}
\numberwithin{equation}{section}
\newcommand{\mycomment}[1]{}
\def\dashmapsto{\mapstochar\dashrightarrow}

\newcommand{\ZZ}{\mathbb Z}
\newcommand{\CC}{\mathbb C}
\newcommand{\PP}{\mathbb P}

\newcommand{\lra}{\longrightarrow}

\newcommand{\cA}{\mathcal{A}}

\newcommand{\tC}{\Tilde{C}}

\newcommand{\cE}{\mathcal{E}}

\newcommand{\s}{\sigma}
\newcommand{\st}{\sigma\tau}
\newcommand{\ee}{\mathcal{E}_2(2,2)}

\DeclareMathOperator{\Aut}{Aut}
\DeclareMathOperator{\id}{id}
\DeclareMathOperator{\Fix}{Fix}
\DeclareMathOperator{\im}{Im}
\DeclareMathOperator{\Nm}{Nm}

\title{Hyperelliptic genus 3 curves with involutions and a Prym map}
\author{Pawe\l{} Bor\'owka and Anatoli Shatsila}
\begin{document}
\maketitle
\begin{abstract}
We characterise genus 3 complex smooth hyperelliptic curves that contain two additional involutions as curves that can be build from five points in $\PP^1$ with a distinguished triple.
We are able to write down explicit equations for the curves and all their quotient curves.

We show that, fixing one of the elliptic quotient curve, the Prym map becomes a 2:1 map and therefore the hyperelliptic Klein Prym map, constructed recently by the first author with A. Ortega, is also 2:1 in this case. 

As a by-product we show an explicit family of $(1,d)$ polarised abelian surfaces (for $d>1$), such that any surface in the family satisfying a certain explicit condition is abstractly non-isomorphic to its dual abelian surface. 
\end{abstract}

\section{Introduction}
Curves with involutions are interesting objects to study. The existence of automorphisms provides a way to find models of curves (e.g. equations for hyperelliptic curves) or periods of Jacobians (e.g. \cite[Section 11.7]{BL}). Moreover, involutions enable to consider (quotient) coverings that are central constructions in the Prym theory.

In the paper, we consider hyperelliptic curves of genus 3 that contain two additional involutions. The motivation to study them came from the Klein covering construction that has been described in \cite{BO17, BO19} and \cite{BO23}. A covering is called Klein if it is Galois 4:1 covering with the deck group being isomorphic to the Klein four-group. A covering is called hyperelliptic if both top and bottom curves are hyperelliptic. In \cite{BO23} the authors show that the Prym map of the hyperelliptic Klein covering is injective in all non-trivial cases (i.e. when genera of curves in the construction increases). 

The covering that has been left in \cite{BO23} is the Klein covering of hyperelliptic genus 3 curve over an elliptic curve, called baby mixed case. It is because in such a case, one has a tower of curves $C\to E'\to E$, where $E,E'$ are elliptic, so $P(C/E)=P(C/E')$, i.e. the Klein Prym variety equals the classical Prym variety. The main result of the paper is the following theorem.

\begin{thm}[Theorem \ref{thmPrym2}]
Let $\mathcal{R}^H_{1,4}$ be a moduli space of hyperelliptic Klein coverings of elliptic curves branched in 4 points and let $\mathcal{E}_d(d,d)$ be the moduli space of $(1,d)$ polarised abelian surfaces containing a pair of complementary elliptic curves of exponent $d$.

Then the considered Prym map $Pr:\mathcal{R}^H_{1,4}\lra \mathcal{E}_2(2,2)$ is finite, dominant, not surjective map of generic degree two. 
\end{thm}
Such result is a little surprising because on the one hand the classical Prym maps of small genus curves are usually not finite and on the other hand all other hyperelliptic Klein Prym maps are injective.

The main idea of the proof is to show that there are precisely two smooth hyperelliptic genus 3 curves on a general abelian surface that belongs to $\mathcal{E}_2(2,2)$.
In order to obtain the above result, we describe the moduli space of genus 3 hyperelliptic curves with two additional commuting involutions in Theorem \ref{moduli}. It turns out that they can be encoded as the 5-tuple of points in $\PP^1$ with the distinguished triple (up to projective equivalence). 

As a by-product of some lemmas needed for the main result, we have been able to show that most of abelian surfaces that belong to $\mathcal{E}_d(d,d)$ for $d>1$ are not isomorphic to their duals:
\begin{thm}[Theorem \ref{dual}]
Let $E,F$ be non-isogenous elliptic curves with $P\in E[d]$ and $Q \in F[d]$ primitive $d$-torsion points. Assume that $E/P$ is not isomorphic to $E$ or $F/Q$ is not isomorphic to $F$. Then $E\times F/\langle(P,Q)\rangle$ (of type $(1,d)$) is not (abstractly) isomorphic to its dual.
\end{thm}

The existence of the involutions implies that the Jacobian of the genus 3 curve is completely decomposable i.e., isogenous to the product of three quotient elliptic curves. Such curves are classically interesting examples, since the decomposition allows us to understand their periods in terms of periods of curves of smaller genera, especially in terms of elliptic curves, see \cite{ES, KR, Paulhus, Paulhus2, eq}. In particular, genus 3 curves with completely decomposable Jacobians have recently gained a lot of attention \cite{HLP, BO19, Katsura, MK}.
Therefore, in Proposition \ref{eqns} we show explicit equations of all the curves appearing in the construction. Since the Prym map is generically 2:1, we are able to write an explicit (rational) deck involution for the Prym map, see Proposition \ref{inv}.

The structure of the paper is as follows. Section 2 is devoted to the theory of non-simple abelian surfaces and the result of dual surfaces being non-isomorphic is proved there. In Section 3 we describe the moduli space and prove the main result concerning the Prym map. 
The last section of the paper is devoted to finding equations of the curves that appear in the construction and describing the deck involution. Moreover, we provide a construction of a period matrix of the Prym variety in terms of periods of elliptic curves. We conclude the paper with some problems for future work.

\subsection*{Acknowledgements}
Some results presented in this article are a part of the second author’s Master Thesis at Jagiellonian University in Kraków. Both authors have been supported by the Polish National Science Center project number 2019/35/D/ST1/02385.
During the preparation of this article
Anatoli Shatsila was a participant of the tutoring programme under the
Excellence Initiative at the Jagiellonian University.



\section{Non-simple abelian surfaces}
We would like to start by recalling some results from the theory of abelian surfaces. Let $\cA_2(1,d)$ be the moduli of $(1,d)$ polarised abelian surfaces and $\cA_2=\cA_2(1,1)$ be the moduli of principally polarised ones. Recall from \cite{BL} and \cite{B16} that for a principally polarised abelian surface we have the following:
\begin{prop}
    Let $A\in\cA_2$ be a principally polarised abelian surface that contains an elliptic curve $E$ of exponent 2. Then there exists a complementary elliptic curve $F$ also of exponent 2 such that the addition map $E\times F\to A$ is a polarised isogeny of degree 4. Note that $E\cap F=E[2]=F[2]\subset A$. 

    Moreover, the locus of principally polarised abelian surfaces containing an elliptic curve of exponent 2 is an irreducible surface in $\cA_2$ that is an image of a so- called singular relation of discriminant $4$.
\end{prop}
\begin{proof}
    The first part is proved in \cite[Cor 12.1.2 and 12.1.4]{BL}. The second part is essentially due to Humbert, see \cite[Thm 4]{B16}.
\end{proof}
Now, we would like to show a similar result for a $(1,2)$ polarised abelian surfaces.
Denote by
$$\mathcal{E}_2(2,2)=\{A\in\cA_2(1,2): A\text{ contains a complementary pair of elliptic curves of exponents } 2,2\}.$$

Recall the notation from \cite{BO17}: for an abelian variety $X$ and its abelian subvarieties $M_1,\ldots,M_k$ with associated symmetric idempotents $\epsilon_1,\ldots,\epsilon_k$, we write $$A=M_1\boxplus\ldots\boxplus M_k\ \text{ if }\ \epsilon_1+\ldots+\epsilon_k=1.$$
In particular, for a pair of complementary elliptic curves $E,F$, we often write $A=E\boxplus F$.
As usual, we also abuse notation by identifying an elliptic curve $E$ with its image under an embedding $E\to A$ unless we consider the product $A = E\times F$. 

The following lemma is a special case of \cite[Thm 5.2]{AB}. 
\begin{lem}
    The locus $\mathcal{E}_2(2,2)$ is a (non-empty) irreducible surface in $\cA_2(1,2)$.  
\end{lem}
\begin{proof}
    Let $X(2)$ be the moduli space of triples $(E,P,P')$, where $P,P' \in E$ is a (symplectic) basis of 2-torsion points. Consider the map 
    
    \begin{align*}
        \Phi: X(2) \times X(2) &\to \mathcal{A}_2(1,2)\\
        ((E, P, P'), (F, Q, Q')) &\mapsto ((E \times F)/\langle (P,Q) \rangle , H),
    \end{align*} where $\Phi^*H = \mathcal{O}_E(2) \boxtimes \mathcal{O}_F(2).$ Since $(P,Q) \in (E \times F)[2]$, the quotient $X = (E \times F)/\langle (P,Q) \rangle$ is a $(1,2)$-polarised abelian surface. Moreover, it follows from \cite[Cor. 6.3.5]{BL} that the quotient map $\pi$ is a polarised isogeny hence $\Phi$ is well-defined. Since $X(2)$ is irreducible, it is enough to show that $\im (\Phi) = \mathcal{E}_2(2,2)$. Since $\pi_{|E}$ is injective, $E$ has exponent 2 in $X$ so the image of $\Phi$ lies in $\mathcal{E}_2(2,2).$ Now, take $(A, H) \in \mathcal{E}_2(2,2)$ and let $(E, F)$ be a complementary pair of elliptic curves in $A$. Poincare's Reducibility Theorem yields an addition map $\mu$ that is a polarised isogeny of degree 2 $$\mu: (E \times F, \mathcal{O}_E(2) \boxtimes \mathcal{O}_F(2)) \to (A, H).$$ By \cite[Prop. 5.3.11]{BL} we know that $\ker \mu = \langle (P,Q) \rangle$ for $P \in E[2], Q \in F[2]$ hence $(A, H) \in \im(\Phi)$. 
\end{proof}

Since we want to show an auxiliary result (i.e. Theorem \ref{dual}), from now on, let $d>1$ and assume that $P\in E[d],\ Q\in F[d]$ are primitive $d$-torsions (i.e. their orders are precisely $d$). In this case $A=(E \times F)/\langle (P,Q) \rangle$ is $(1,d)$ polarised and $A\in\cE_d(d,d)$.
We have the following.
\begin{prop}\label{noniso}
If $E,F$ are non-isogenous then $A$ does not contain any other elliptic curve. 
\end{prop}
\begin{proof}
   Let $G$ be en elliptic curve contained in $A$ and consider the Norm map of $E$. 
    Note that either $\im(\Nm_E|_G)=\{0\}$ or $\im(\Nm_E|_G)=E$ and the latter occurs when $E$ and $G$ are isogenous.  
Now, recall that $m_d\Nm_E+m_d\Nm_F=m_{d^2}$ where $m_k$ is multiplication by $k$ (see \cite[p. 124]{BL}), so $\im(\Nm_E|_G+\Nm_F|_G)=G$, hence if one of the images is $0$ then $G$ is equal to the other curve. However, if both maps are non-zero then $G$ is isogenous to both curves, hence $E,F$ are isogenous.
\end{proof}
\begin{lem}\label{lemintersections}
Let $E,F$ be complementary elliptic curves in $A=(E\times F)/\langle (P,Q) \rangle$ of exponent d. Then $Q=-P\in A$ and $E\cap F=\langle P \rangle\subset \ker\phi_H$. Moreover $E\cap \ker\phi_H=F\cap \ker\phi_H=\langle P \rangle$.     
\end{lem}
\begin{proof}
Recall that $\mu(P,0)=P\in E,\ \mu(0,Q)=Q\in F$. Now, the first equation follows trivially from the fact that $(P,0)+(0,Q)=(P,Q)\in \ker\mu$.

To see that $E\cap F=\langle P \rangle$ one can check that $\mu(x,0)=\mu(0,y)\iff (x,-y)\in\ker\mu=\langle (P,Q) \rangle$ and use the fact that $P,Q$ are primitive $d$-torsions.

Now, in order to finish the proof, one needs to use the fact that $\ker\phi_H=\mu((\ker\mu)^{\perp})$ where $\perp$ stands for the symplectic complement (with respect to the symplectic pairing $\omega$), see \cite[Prop 6.2.1 and 6.3.3]{BL}. 
Since $x\in E[d]$ is of the form $aP+bP'$ (where $a,b\in\ZZ_d$) with $\omega(P,P')=1\in\ZZ_d$ we get that $(aP+bP',0)\in (\ker\mu)^{\perp} \iff b=0$.  
\end{proof}

Now, we are ready to show a way to find an explicit condition that proves that a surface is not abstractly isomorphic to its dual. 
\begin{thm}\label{dual}
Let $E,F$ be non-isogenous elliptic curves with $P\in E[d]$ and $Q \in F[d]$ primitive $d$-torsion points. Assume that $E/P$ is not isomorphic to $E$ or $F/Q$ is not isomorphic to $F$. Then $E\times F/\langle(P,Q)\rangle$ (of type $(1,d)$) is not (abstractly) isomorphic to its dual.
\end{thm}
\begin{proof}
By symmetry we can assume $E/P$ is not isomorphic to $E$.
    Denote by $A=E\times F/\langle(P,Q)\rangle$. Then its dual is $\hat{A}=A/\ker\phi_H$.
    By Lemma \ref{lemintersections}, we get that $E\cap \ker\phi_H=\langle P\rangle$, so $E'=E/\langle P\rangle$ is an elliptic curve embedded in $\hat{A}$. If $\hat{A}$ was isomorphic to $A$ then the image of $E'$ via the isomorphism would be an elliptic curve embedded in $A$. By assumption it cannot be $E$ and it cannot be $F$ since they are non-isogenous. Since this contradicts Proposition \ref{noniso}, we get that $\hat{A}$ is not isomorphic to $A$. 
\end{proof}

\begin{rem}
    It is well-known that a general non-principally polarised abelian variety is not isomorphic to its dual. The advantage of Theorem \ref{dual} is that it provides a family of examples that enables to write down  explicit period matrices of surfaces.
\end{rem}
We would like to finish this section by describing better the connection between $A$ and $\hat{A}$.
\begin{prop}\label{duality}
If $A=E\times F/\langle(P,Q)\rangle$, then $A/P=E'\times F'$ and
$\hat{A}=E'\times F'/G$, where $G\simeq \ker\phi_H/P\simeq \langle(P',-Q')\rangle$. 
In particular, $A\in\ee \iff\hat{A}\in\ee$.
\end{prop}
\begin{proof}
    Note that $A/P = (E \times F)/\langle (P, 0), (0, Q)\rangle$ and $\langle (P, 0), (0, Q)\rangle$ is an isotropic subgroup of the form $\ZZ_d\oplus\ZZ_d$ coming from $\mathcal{O}_E(d) \boxtimes \mathcal{O}_F(d)$. Thus, $A/P = E' \times F'$ (with a principal polarization) and the quotient map is a polarised isogeny by \cite[Cor. 6.3.5]{BL}. Since $\hat{A} = A/ \ker\phi_H$, we get the desired equality for $\hat{A}$. In the case $d = 2$, $(P', -Q')$ is a 2-torsion point of $E'\times F'$, hence $\hat{A}$ lies in the image of $\Phi$, which finishes the proof.  
\end{proof}

\section{Geometry of hyperelliptic genus 3 curves with involutions}
Now, we are ready to set up the construction. We start by recalling a useful result concerning involutions on hyperelliptic curves that is an easy application of Accola's Lemma \cite{accola}. 
\begin{lem}
\label{l1m}
    Let $X$ be a smooth hyperelliptic curve of genus $g = 3$ with $\iota$ the hyperelliptic involution. For any $\tau \in \Aut(X)$, $\tau^2 = \id$, $\tau \neq \iota$ we have $$\Fix(\tau) = \varnothing, \: |\Fix(\iota\tau)| = 4, \: \: \: \text{or} \:\:\: |\Fix(\tau)| = 4, \Fix(\iota\tau) = \varnothing.$$

    Now, let $\sigma \neq \tau$ be commuting involutions on $X$ such that $|\Fix(\sigma)| = |\Fix(\tau)| = 4$. Then $|\Fix(\sigma\tau)| = 4$ and $g(X/\langle\sigma,\tau\rangle)=0$. 
\end{lem}

\begin{proof}

By Accola's Lemma (\cite[Lemma 1]{accola}) applied to $G_0 = \langle \tau, \iota\rangle$ we get $$3 = g(X/\tau) + g(X/\iota\tau).$$ Then, Hurwitz formula yields $g(X/\tau) = 2$ and $g(X/\iota\tau) = 1$ or vice versa. 
    
    For the second part, applying Accola's Lemma (\cite{accola}) to $G_0 = \langle \tau, \sigma \rangle$ we have
    $$3 + 2g(X/\langle\sigma,\tau\rangle) = g(X/\sigma) + g(X/\tau) + g(X/\sigma\tau).$$
    By Hurwitz formula, $g(X/\sigma) = g(X/\tau) = 1$, hence $g(X/\sigma\tau)$ is odd. The only possibility is $g(X/\sigma\tau) = 1$, hence $|\Fix(\sigma\tau)| = 4$. Since the right hand side of the equation equals 3, we get that $g(X/\langle\sigma,\tau\rangle)=0$.
\end{proof}

From now on, we will view elliptic curves both as one-dimensional abelian varieties and genus one curves with a chosen point. We will also use the fact that any elliptic curve is canonically isomorphic to its dual (as an abelian variety).
With such convention, we state  the following result.
\begin{lem}\label{lem32}
   Let $C$ be a smooth curve of genus $g>1$ and $E$ an elliptic curve. Then the following are equivalent:
   \begin{itemize}
       \item[i)] there exists a double covering $f:C\to E$; 
       \item[ii)] $E$ can be embedded into $JC$ as an abelian subvariety of exponent $2$.
   \end{itemize}
\end{lem}
\begin{proof}
    The equivalence follows from the fact that $f$ does not factorise via a cyclic \'etale covering (because it is branched of degree 2), hence $f^* \circ Nm_f=Nm_E$, see \cite[Prop 12.3.2]{BL}.
\end{proof}

In \cite{BO23} the authors have investigated hyperelliptic curves with additional commuting involutions and they have shown the injectivity of the hyperelliptic Klein Prym map in a general case. The only case that has been left is a genus 3 hyperelliptic curve that is a Klein covering of an elliptic curve. Now, we are ready to investigate such a case.

 Let $\Tilde{C}$ be a hyperelliptic genus 3 curve such that $(\mathbb{Z}/2\mathbb{Z})^3 \simeq \langle \iota, \sigma, \tau \rangle \subseteq \Aut(\Tilde{C})$. Here $\iota$ denotes the hyperelliptic involution on $\Tilde{C}$. In addition, we assume that $\sigma$ and $\tau$ each have 4 fixed points and hence $\sigma\tau$ also has 4 fixed points (see Lemma \ref{l1m}). We also get $|\Fix(\iota\tau)| = |\Fix(\iota\sigma)| = |\Fix(\iota\sigma\tau)| = 0$.
We introduce the following notation for quotient curves: 
\begin{align*}
    E_{\sigma} := \Tilde{C}/\langle \sigma \rangle, & \ \ \ \ 
    E_{\tau} := \Tilde{C}/\langle \tau \rangle,\\ 
    E_{\sigma\tau} := \Tilde{C}/\langle \sigma\tau \rangle,& \ \ \ \
    E_{\sigma,\iota\tau} := \Tilde{C}/\langle \sigma, \iota\tau \rangle,\\ 
    E_{\tau, \iota\sigma} := \Tilde{C}/\langle \tau, \iota\sigma \rangle, & \ \ \ \ 
    E_{\iota\sigma, \iota\tau} := \Tilde{C}/\langle \iota\sigma, \iota\tau \rangle, \\
    C_{\iota\sigma} := \Tilde{C}/\langle \iota\sigma \rangle, & \ \ \ \ 
    C_{\iota\tau} := \Tilde{C}/\langle \iota\tau \rangle,\\ C_{\iota\sigma\tau} := \Tilde{C}/\langle \iota\sigma\tau \rangle.&
\end{align*}
Note that Hurwitz formula yields that the six curves denoted by letter $E$ are elliptic, while the other three denoted by $C$ have genus 2. We have the following commutative diagram, where by +4 (resp. +2) we indicate that a degree two map is branched at 4 points (resp. at 2 points). 

$$
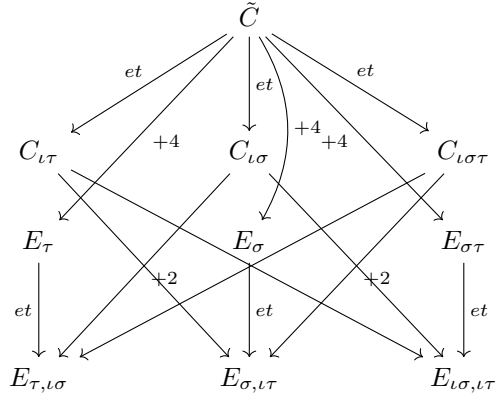
\begin{tikzcd} \label{picture}
                                                  &  & \Tilde{C} \arrow[lldd, "et"'] \arrow[dd, "et"] \arrow[rrdd, "et"] \arrow[llddd, "+4"] \arrow[rrddd, "+4"'] \arrow[ddd, "+4", bend left] &  &                                                         \\
                                                  &  &                                                                                                                                         &  &                                                         \\
C_{\iota\tau} \arrow[rrddd] \arrow[rrrrddd] &  & C_{\iota\sigma} \arrow[llddd, "+2"] \arrow[rrddd]                                                                                             &  & C_{\iota\sigma\tau} \arrow[llddd, "+2"] \arrow[llllddd] \\
E_{\tau} \arrow[dd, "et"']                        &  & E_{\sigma} \arrow[dd, "et"]                                                                                                             &  & E_{\sigma\tau} \arrow[dd, "et"]                         \\
                                                  &  &                                                                                                                                         &  &                                                         \\
{E_{\tau,\iota\sigma}}                            &  & {E_{\sigma,\iota\tau}}                                                                                                                  &  & {E_{\iota\sigma,\iota\tau}}                            
\end{tikzcd}$$

Our aim is to investigate the diagram above. We start by showing how to associate five points on $\mathbb{P}^1$ (with a distinguished triple) to a curve $\Tilde{C}$. Let $w, \s (w), \tau (w),\st (w), w', \s (w'), \tau (w'), \st (w')$ be the eight Weierstrass points of $\Tilde{C}$ and 
\begin{align*}
    \Fix(\sigma) &= \{p_1, p_2, \iota p_1, \iota p_2\}, \\
    \Fix(\tau) &= \{q_1, q_2, \iota q_1, \iota q_2\}, \\
    \Fix(\sigma\tau) &= \{r_1, r_2, \iota r_1, \iota r_2\}.
\end{align*}
A proof that all points defined above are distinct from each other can be found in \cite{BO23}.

Let $f: \Tilde{C}\to \Tilde{C}/\langle \sigma,\tau \rangle \simeq \mathbb{P}^1$ be the quotient $4:1$ map and $[w]$ and $[w']$ be the images of Weierstrass points under $f$. Let $g: \mathbb{P}^1 \to \mathbb{P}^1$ be a 2:1 map ramified at $[w]$ and $[w']$ and $[p], [q], [r]$ be the images under $g\circ f$ of fixed points of $\sigma, \tau$ and $\sigma\tau$ respectively. Finally, let $h: \mathbb{P}^1 \to \mathbb{P}^1$ be a 4:1 map branched at $[p], [q], [r]$ with two simple ramifications at each fiber and $\pi_{\tC}:\tC\to\PP^1$ be the hyperelliptic $2:1$ map. In this way we have constructed the following commutative diagram

$$\begin{tikzcd}
\Tilde{C} \arrow[dd, "\pi_{\Tilde{C}}"'] \arrow[rr, "f"] &  & \mathbb{P}^1 \arrow[dd, "g"] \\
                                                           &  &                                \\
\mathbb{P}^1 \arrow[rr, "h"]                             &  & \mathbb{P}^1                  
\end{tikzcd}$$
where $$h \circ \pi_{\Tilde{C}} = g \circ f: \Tilde{C} \to \Tilde{C}/\langle \sigma, \tau, \iota \rangle = \mathbb{P}^1$$ is the $8:1$ quotient map. We have a 5-tuple of points lying in the bottom-right $\mathbb{P}^1$, namely $$(g([w]), g([w']); [p], [q], [r]).$$ This is the 5-tuple associated to $\Tilde{C}$. Here, $[p], [q]$ and $[r]$ form a distinguished triple. 

On the other hand, starting with a 5-tuple with a distinguished triple $$(x,y; [p], [q], [r])$$ we can reconstruct the curve $\Tilde{C}$ as follows. Take a $\mathbb{Z}_2^2$-Galois covering branched at $[p], [q], [r]$. The preimages of $x,y$ give us 8 points $w_1,\ldots, w_8$. Then, $\Tilde{C}$ is a double cover of $\mathbb{P}^1$ branched at these points. Since $w_1,\ldots, w_8$ are invariant under the action of $\mathbb{Z}_2^2$, there are two additional commuting involutions on $\tC$ lifted from $\PP^1$.

\begin{thm}\label{moduli}
The following are equivalent:
\begin{enumerate}
    \item the locus of 5 points in $\PP^1$ with a chosen triple (up to projective equivalence respecting the distinguished triple);
    \item the locus of hyperelliptic genus 3 curves having two additional commuting involutions (different from the hyperelliptic one);
    \item the locus of genus 3 curves $\Tilde{C}$ having $\ZZ^3_2\subset Aut(\Tilde{C})$;
    \item the locus of hyperelliptic curves given by $y^2=(x^4+ax^2+1)(x^4+bx^2+1)$, $a,b\in\CC,\ a\neq b$ $a^2,b^2\neq 4$.
\end{enumerate}
\end{thm}
\begin{proof}
    The equivalence $(1) \Leftrightarrow (2)$ follows from the construction. The implication $(2) \Rightarrow (3)$ is obvious while $(3) \Rightarrow (2)$ comes from Accola's lemma (\cite{accola}). The equivalence $(2) \Leftrightarrow (4)$ follows from Proposition \ref{mainpeq} that is proved independently.
\end{proof}
\begin{rem}
    Note that irreducibility of the moduli space and its dimension being equal to two follows easily from point 4 of Theorem \ref{moduli}. One can also see Theorem \ref{moduli} as a characterisation of the moduli space of hyperelliptic curves that have a completely decomposed
Richelot isogeny, see \cite{Katsura}.
\end{rem}

\subsection{From bottom to up and a Prym map}
Now, we would like to describe the construction starting from one of elliptic curves, say $E_\s$.
We start with a characterisation lemma that is similar to Theorem \ref{moduli}.

\begin{lem}
\label{el4p}
The following are equivalent:
\begin{enumerate}
    \item the locus of 5 points in $\mathbb{P}^1$ with a distinguished triple and a point in the triple (up to projective equivalence respecting the triple and the chosen point); 
    \item the locus of pairs $(\tC,\sigma)$, where $\ZZ_2^3\subset \Aut(\tC)$  and $\sigma$ is an involution with $|\Fix(\sigma)|=4$;
    \item the locus of pairs $(E,B)$ where E is an elliptic curve and $B$ is a set of 4 points invariant under both $(-1)$ and a translation by a 2-torsion point (and $B\cap E[2]=\emptyset$);
    \item the locus of genus 2 curves $C$ that are double coverings of elliptic curves together with a pair of Weierstrass points on $C$ whose images under both covering maps are equal.
\end{enumerate}
\end{lem}
\begin{proof}
     The equivalence $(1) \Leftrightarrow (2)$ follows from the construction, since there is a one to one correspondence between involutions on $\Tilde{C}$ with four fixed points in the case $(2)$ and the points of the distinguished triple in the case $(1)$. 
    
    To obtain $(2) \Rightarrow (3)$ we take $E := E_{\sigma}$ and $B$ equal to the branch locus of the covering defined by $\sigma$. To show $(3) \Rightarrow (2)$ one considers the covering branched at $B$ and defined by the line bundle $\eta=\mathcal{O}_E(2O)$. Similarly, to get $(2) \Rightarrow (4)$ we take $C = \Tilde{C}/\iota\sigma$. We leave the details of the implications to the reader since they are straightforward applications of the construction.
    
    It remains to show $(4) \Rightarrow (2)$. Let $w_1, w_2$ be two Weierstrass points on $C$ whose images under the coverings of elliptic curves are equal. We define $\Tilde{C} \to C$ to be an \'etale hyperelliptic covering defined by $\mathcal{O}_C(w_1 - w_2)$. Let $\s$ be an involution of $\Tilde{C}$ exchanging the sheets of the covering. Then $(\Tilde{C}, \iota \s)$ is the corresponding pair, because the covering involution of $C/E$ can be lifted to an involution on $\tC$ that commutes with $\s$. 
\end{proof}

\begin{rem}
    Note that changing a distinguished point within a triple means that one starts with another elliptic curve (another involution) but the genus 3 curve $\tC$ is the same.
\end{rem}

Using the fact that $\Tilde{C}/\langle \sigma,\tau \rangle \simeq \mathbb{P}^1$ we get that $$J\tC=E_\s\boxplus E_\tau\boxplus E_{\st}.$$

Recall that the Prym variety of the covering $f:X\to Y$ can be defined as the complementary abelian subvariety to $f^*(JY)$ in $JX$ and is usually denoted by $P(f)$ or $P(X/Y)$, see \cite[Chapter 12]{BL}. With such notation, we can write $$JX=f^*(JY)\boxplus P(X/Y).$$

In our case, by choosing $\s$ we get that the Prym variety of the covering $\tC\to E_\s$ is equal to $P(\tC/E_\s)=E_\tau\boxplus E_{\st}$. Therefore, since the Prym variety is complementary to the elliptic curve of exponent 2 and by Lemma \ref{lem32}, we get that  $$P(\tC/E_\s)\in\mathcal{E}_2(2,2).$$ 
\begin{rem}
    Note that $P(\tC/E_\s)=P(\tC/E_{\s,\iota\tau})$, hence we also study Pryms of hyperelliptic Klein coverings.
\end{rem}

The main point of the Prym theory is to consider \textit{the Prym map} that assigns to a covering its Prym variety, see for example \cite{Be, BOprym}. 
Let $$\mathcal{R}^{H}_{1,4} = \{(\Tilde{C},\sigma) \: | \: g(\Tilde{C}) = 3,  (\mathbb{Z}/2\mathbb{Z})^3\subset \Aut(\Tilde{C}),   \sigma^2 = \id,  |\Fix(\sigma)|=4\}$$ be the moduli space considered in Lemma \ref{el4p}.
It follows that we have a well-defined Prym map $$\Pr: \mathcal{R}^H_{1,4} \to \mathcal{E}_2(2,2).$$  
We want to find the generic degree of $\Pr$. We start by recalling some results.

\begin{prop}[\cite{Barth}, Duality Theorem 1.12]
\label{bar}
Let $C$ be a hyperelliptic genus 3 curve  and let $A$ be a $(1,2)$ polarised abelian surface. Then the following are equivalent:
\begin{enumerate}
    \item $C$ can be embedded in $\hat{A}$;
    \item $A$ can be embedded in $JC$;
    \item $C$ is a double covering of an elliptic curve that is complementary to $A$ in $JC$.
\end{enumerate}
\end{prop}

Note that by \cite[Prop 6]{B16}, there are precisely three smooth hyperelliptic curves on a general $(1,2)$ polarised abelian surface. Now, we are ready to show a similar result in our case.
\begin{prop}\label{twocurves}
    On $A\in\ee$ there are at most two smooth hyperelliptic genus 3 curves (up to a translation) and there is an open dense set $U\subset\ee$ such that for $A\in U$ there are precisely two smooth curves on $A$.
\end{prop}
\begin{proof}
 Firstly, by \cite[Prop 6]{B16}, there is a 1:1 correspondence between smooth hyperelliptic curves on $A$ and polarised isogenies to Jacobians of smooth curves and there are at most three of them.

 Recall that a principally polarised abelian surface is either a Jacobian of a smooth curve or a polarised product of two elliptic curves. Hence if $\pi:A\to E'\times F'$ is a polarised isogeny of degree 2, then $\pi^{-1}(E'\times \{0\})$ and $\pi^{-1}(\{0\}\times F')$ are elliptic curves in $A$. By Proposition \ref{noniso} there are precisely two elliptic curves on a general $A\in \ee$.

Moreover, if $A=E\times F/\langle(R,Q)\rangle$, then $A/\langle R\rangle =E'\times F'$, so indeed there is at least one product of elliptic curves and at most two Jacobians of smooth curves downstairs.  

Let $U=\{A: \exists P,Q\in A[2], A/P \text{ and } A/Q \text{ are Jacobians of smooth curves}\}$. 
Note that $U$ is open and dense because the set of Jacobians is open and dense.
\end{proof}

For low genera, the Prym map is usually dominant and with positive dimensional fibres for dimensional reasons. Since we are interested in special coverings, both domain and codomain have the same dimension, so we can consider the (generic) degree of the Prym map.
Now we would like to compute the degree of the Prym map.

\begin{thm}\label{thmPrym2}
The considered Prym map $Pr:\mathcal{R}^H_{1,4}\lra \mathcal{E}_2(2,2)$ is finite, dominant, not surjective map of generic degree two. 
\end{thm}
\begin{proof}
Let $P\in \mathcal{E}_2(2,2)$. By Proposition \ref{duality}, we get that $\widehat{P}\in\ee$, so by Proposition \ref{twocurves} there exists at most 2 smooth genus 3 hyperelliptic curves (precisely 2 on a general $P$), called $C_i$ (that are coverings of $E_i$) that are embedded into $\widehat{P}$. 
On the other hand, Proposition \ref{bar} says that this is a necessary and sufficient condition for $P$ to be the Prym of a covering $P(C_i/E_i)$.
The fact that $Pr$ is not surjective follows from Example \ref{surj}.
\end{proof}

\begin{exam}
\label{surj}
    There exists $A \in \mathcal{E}_2(2,2)$ not in the image of $Pr$.
\end{exam}

\begin{proof}
    Let $E$ be an elliptic curve given by the equation $y^2 = x^3 + x$. Then $\Aut(E) = \langle \alpha \rangle \simeq \mathbb{Z}/4\mathbb{Z}$. Let $E_1=E_2=E$. We denote non-zero two-torsion points on $E_i$ by $e_i, f_i, e_i + f_i$ in such a way that $$\alpha(e_i) = e_i,\ \alpha(f_i) = e_i + f_i,\ \alpha(e_i + f_i) = f_i.$$

    We denote four elliptic curves isomorphic to $E$ and lying in $E_1 \times E_2$ as follows: \begin{align*}
        E_\Delta &:= \{(s,s): s \in E\}; \\
        E_{-\Delta} &:= \{(s,-s): s \in E\}; \\
        E_{\alpha} &:= \{(s,\alpha(s)): s \in E\}; \\
        E_{-\alpha} &:= \{(s,-\alpha(s)): s \in E\}. 
    \end{align*}

        Let $A = (E_1 \times E_2)/\langle (e_1,e_2) \rangle \in \mathcal{E}_2(2,2)$ and $\mu$ be the quotient map. We abuse the notation by writing $\mu(x_1)=\mu(x_1,0)$ and $\mu(x_2)=\mu(0,x_2)$. In this way we can write $\mu(e_1)=\mu(e_2)\in A$. In particular, we get that that $\mu(E_1)\cap \mu(E_2)=\{0,\mu(e_1)\}$.  

    Note that $\mu(E_\Delta)$ and $\mu(E_{-\Delta})$ are isomorphic to $E/\langle e \rangle$ and are complementary in $A$ of exponent 2, and $\mu(E_
    \Delta) \cap \mu(E_{-\Delta}) = \{0, \mu(f_1,f_2)\}$. Analogously, $\mu(E_\alpha)$ and $\mu(E_{-\alpha})$ are isomorphic to $E/\langle e \rangle$ and complementary in $A$ of exponent 2, and $\mu(E_\alpha) \cap \mu(E_{-\alpha}) = \{0, \mu(e_1 + f_1, f_2)\}$.
    
    Recall that we compute the kernel of the polarising isogeny by $\ker\phi_A=\mu((\ker\mu)^{\perp})$, hence $\ker \phi_A = \{0, \mu(e_1), \mu(f_1,f_2), \mu(e_1 + f_1,f_2)\}$. But we have \begin{align*}
        A/\langle \mu(e_1) \rangle &= E_1/\langle e_1 \rangle \times E_2/\langle e_2 \rangle, \\
        A/\langle \mu(f_1,f_2) \rangle &= E_\Delta/E_\Delta[2] \times E_{-\Delta}/E_{-\Delta}[2], \\
        A/\langle \mu(e_1+f_1,f_2) \rangle &= E_3/E_3[2] \times E_4/E_4[2],
    \end{align*}
    where all products are principally polarized. Since none of the quotients is a Jacobian of a smooth genus 2 curve, it follows that $A$ is not in the image of $\Pr$.
\end{proof}

\begin{rem}
\label{ramif}
By \cite{B16}, a general $A\in\cA_(1,2)$ contains three smooth hyperelliptic genus 3 curves.
By Proposition \ref{twocurves}, a general $A=E\boxplus F\in\cE_2(2,2)$ contains two smooth hyperelliptic genus 3 curves and a singular curve $E\cup F$ of arithmetic genus 3.
One can use the proof of Example \ref{surj} to show that on $E\boxplus E$ there are two singular curves $E\cup E$ and $E_{\Delta}\cup E_{-\Delta}$ of arithmetic genus 3, hence the degree of the Prym map at $E\times E/\langle(R,R)\rangle$ is at most one.
In the special case when $E$ has automorphism of order 4, we have got three singular curves of arithmetic genus 3 and no smooth genus 3 hyperelliptic curves, so the Prym map is not surjective.

\end{rem}

\section{Equations}
\label{eqeq}
In this section we find equations of curves that appear in the construction. We start with a useful lemma.

\begin{lem}
\label{eqg2}
    Let $C$ be a smooth hyperelliptic curve of genus 2 with a non-hyperelliptic involution $\sigma$. Then $C$ is isomorphic to a curve given by the equation $$y^2 = (x^2 -1 )(x^2 - t_1^2)(x^2 - t_2^2)$$ for some $t_1, t_2 \in \mathbb{C}$ satisfying $t_1 \neq \pm t_2, \ t_1^2, t_2^2 \neq 1$. 
\end{lem}

\begin{proof}
    
Since $C$ is hyperelliptic, the involution $\sigma$ induces an involution $\overline{\sigma}$ of $\mathbb{P}^1$ and we have a commutative diagram
\begin{equation*}
\begin{tikzcd}[row sep=huge]
C \arrow[r,"\sigma"] \arrow[d,swap,"\pi"] & C \arrow[d,swap,"\pi"]
\\
\mathbb{P}^1 \arrow[r,"\overline{\sigma}"] & \mathbb{P}^1 
\end{tikzcd}
\end{equation*}
By Hurwitz’s formula applied to $\pi_{\overline{\sigma}}: \mathbb{P}^1 \to \mathbb{P}^1/\overline{\sigma} \simeq \mathbb{P}^1$, we see that $\overline{\sigma}$ has two fixed points $p, q \in \mathbb{P}^1$. Let $r = \pi(w_1)$ be an image of some Weierstrass point on $C$. By changing the coordinates on $\mathbb{P}^1$ we can assume that $p = [1:0], q = [0:1]$ and $r = [1:1]$. It follows that $\overline{\sigma}([x:y]) = [-x: y]$. Without loss of generality, we can assume that $\sigma(x,y) = (-x, y)$ and, since $\sigma$ preserves Weierstrass points, $C$ has the defining equation $$y^2 = (x^2 - 1)(x^2 - t_1^2)(x^2 - t_2^2)$$ for some $t_1, t_2 \in \mathbb{C}$ satisfying $1 \neq t_1^2 \neq t_2^2 \neq 1$.
\end{proof}

\begin{prop}
\label{mainpeq} 
    Let $\Tilde{C}$ be a smooth hyperelliptic curve of genus $3$ such that $\Aut(\Tilde{C}) = \langle \sigma, \tau, \iota \rangle$ with $|\Fix(\sigma)| = |\Fix(\tau)| = 4$ and $\iota$ the hyperelliptic involution. Then $\Tilde{C}$ is isomorphic to a smooth hyperelliptic curve given by the equation $$y^2 = (x^4 + ax^2 + 1)(x^4 + bx^2 + 1),$$ where $(a, b) \in (\mathbb{C} \setminus \{-2, 2\})^2 \setminus \Delta$. 
\end{prop}

\begin{proof}
    Let $\overline{\sigma}$ and $\overline{\tau}$ be the involutions of $\mathbb{P}^1$ induced by $\sigma$ and $\tau$ respectively. As before, we can assume that $\overline{\sigma}([x:y]) = [-x:y]$. 
    It is easy to see that any involution $j$ of $\mathbb{P}^1$ is given by $j([x:y]) = [ax + by: cx - ay]$ with $a^2 + bc = 1$. Thus, we can assume that $\overline{\tau}$ is of this form. Now, since $\overline{\tau}$ and $\overline{\sigma}$ commute, we have $\overline{\tau}([x:y]) = [ky:x]$, for some $k \in \mathbb{C}$. By Lemma \ref{l1m} and Hurwitz formula we get $g(\Tilde{C}/\iota\tau) = 2$, and the following diagram commutes:
\begin{equation*}
\label{diag1}
\begin{tikzcd}[row sep=huge]
\Tilde{C} \arrow[r,"\sigma"] \arrow[d,swap,"\pi_{\iota\tau}"]& \Tilde{C} \arrow[d,swap,"\pi_{\iota\tau}"]
\\
\Tilde{C}/\iota\tau \arrow[r,"\sigma'"] \arrow[d,swap,"\pi_{\Tilde{C}/\iota\tau}"] & \Tilde{C}/\iota\tau \arrow[d,swap,"\pi_{\Tilde{C}/\iota\tau}"]
\\
\mathbb{P}^1/\overline{\tau} \arrow[r,"\overline{\sigma}(
mod \overline{\tau})"]& \mathbb{P}^1/\overline{\tau}
\end{tikzcd}
\end{equation*}

By Lemma \ref{eqg2}, we can assume that $\Tilde{C}/\iota\tau$ has the defining equation $$y^2 = (x^2 - 1)(x^2 - t_1^2)(x^2 - t_2^2)$$ over $\mathbb{P}^1/\overline{\tau} \simeq \mathbb{P}^1$ and, moreover, the preimages $p_1, p_2$ under $\pi_{\iota\tau}$ of the Weierstrass point $(1, 0)$ on $\Tilde{C}/\iota\tau$ are not the Weierstrass points on $\Tilde{C}$. Then $\pi_{\Tilde{C}}(p_1) = \pi_{\Tilde{C}}(p_2) = [1:1]$ is a fixed point of $\overline{\tau}$, hence $k = 1$. It follows that the Weierstrass points on $\Tilde{C}$ are $(\pm t_1, 0), (\pm t_2, 0), (\pm \frac{1}{t_1}, 0), (\pm \frac{1}{t_2}, 0)$ for some $t_1, t_2 \in \mathbb{C}$, hence the defining equation of $\Tilde{C}$ is $$y^2 = (x^4 + ax^2 + 1)(x^4 + bx^2 + 1)$$ with $a = -t_1^2 - \frac{1}{t_1^2}, b = - t_2^2 - \frac{1}{t_2^2}$. In order for $\Tilde{C}$ to be smooth the Weierstrass points must be pairwise distinct. Thus, we have $(a, b) \in (\mathbb{C} \setminus \{-2, 2\})^2 \setminus \Delta$. 
\end{proof}

The following proposition can be seen as a first step to see the construction in coordinates.
\begin{prop}
\label{eqns} 
    Let $\Tilde{C}$ be a hyperelliptic genus 3 curve given by the equation $$y^2 = (x^4 + ax^2 + 1)(x^4 + bx^2 + 1)$$ and $\sigma, \tau$ be the involutions of $\Tilde{C}$ such that the involutions $\overline{\sigma}, \overline{\tau}$ of the projective line induced by $\sigma$ and $\tau$ are given by $$\overline{\sigma}([x:y]) = [-x:y], \:\:\: \overline{\tau}([x:y]) = [y:x].$$ 
    Then the quotient curves have the following defining equations:
    \begin{align*}
    C_{\iota\tau}: \:\:\: y^2 &= (x^2 - 4)(x^2 + a - 2)(x^2 + b - 2), \\
    C_{\iota\sigma}: \:\:\: y^2 &= x(x^2 + ax + 1)(x^2 + bx + 1), \\
    C_{\iota\sigma\tau}: \:\:\: y^2 &= (x^2 + 4)(x^2 + a + 2)(x^2 + b + 2), \\
    E_{\tau}: \:\:\: y^2 &= (x^2 + a - 2)(x^2 + b - 2), \\
    E_{\sigma}: \:\:\: y^2 &= (x^2 + ax + 1)(x^2 + bx + 1), \\
    E_{\sigma\tau}: \:\:\: y^2 &= (x^2 + a + 2)(x^2 + b + 2), \\
    E_{\iota\sigma, \iota\tau}: \:\:\: y^2 &= (x + a)(x + b)(x - 2), \\
    E_{\iota\tau, \sigma}: \:\:\: y^2 &= (x + a)(x + b)(x - 2)(x + 2), \\
    E_{\iota\sigma, \tau}: \:\:\: y^2 &= (x + a)(x + b)(x + 2).
    \end{align*}    
\end{prop}

\begin{proof}
    In the proof of Proposition \ref{mainpeq} we found the equation of $C_{\iota\tau}$ over $\mathbb{P}^1/{\overline{\tau}}$. We want to find its equation over $\mathbb{P}^1$. The isomorphism $\varphi: \mathbb{P}^1/{\overline{\tau}} \to \mathbb{P}^1$ is given by $\varphi([x:1]) = [x + \frac{1}{x}:1]$ for $x \neq 0$ and $\varphi([0:1]) = [1:0]$, hence the defining equation of $C_{\iota\tau}$ over $\mathbb{P}^1$ is given by $$y^2 = (x^2 - 4)\left(x^2 - \left(t_1 + \frac{1}{t_1}\right)^2\right)\left(x^2 - \left(t_2 + \frac{1}{t_2}\right)^2\right) = (x^2 - 4)(x^2 + a - 2)(x^2 + b - 2).$$

    From the commutative diagram below we see that the Weierstrass points on $E_{\sigma}$ are the images under $\pi_{\sigma}$ of points in $\Fix(\iota) \cup \Fix(\iota\sigma) = \Fix(\iota)$ since $\iota\sigma$ is fixed-point free. 
\begin{equation*}
\begin{tikzcd}[row sep=huge]
C \arrow[r,"\iota"] \arrow[d,swap,"\pi_{\sigma}"] & C \arrow[d,swap,"\pi_{\sigma}"]
\\
E_{\sigma} \arrow[r,"\iota_{C}"] & E_{\sigma}
\end{tikzcd}
\end{equation*}
Thus, $E_{\sigma}$ has the following defining equation over $\mathbb{P}^1/{\overline{\sigma}}:$
$$y^2 = (x - t_1)\left(x - \frac{1}{t_1}\right)(x - t_2)\left(x - \frac{1}{t_2}\right).$$ Transforming the equations by $\varphi: \mathbb{P}^1/{\overline{\sigma}} \to \mathbb{P}^1$ with $\varphi([x:y]) = [x^2:y]$ we get the equation $$y^2 = (x - t_1^2)\left(x - \frac{1}{t_1^2}\right)(x - t_2^2)\left(x - \frac{1}{t_2^2}\right) = (x^2 + ax + 1)(x^2 + bx + 1).$$
Analogously, the defining equations of $E_{\tau}$ and $E_{\sigma\tau}$ are given by $$y^2 = \left(x^2 - \left(t_1 + \frac{1}{t_1}\right)^2\right)\left(x^2 - \left(t_2 + \frac{1}{t_2}\right)^2\right) = (x^2 + a - 2)(x^2 + b - 2)$$ and $$y^2 = \left(x^2 - \left(t_1 - \frac{1}{t_1}\right)^2\right)\left(x^2 - \left(t_2 - \frac{1}{t_2}\right)^2\right) = (x^2 + a + 2)(x^2 + b + 2)$$
respectively.

The Weierstrass points of $C_{\iota\sigma}$ are the images under $\pi_{\iota\sigma}$ of points in $\Fix(\iota) \cup \Fix(\sigma)$. That is, $$W(C_{\iota\sigma}) = 
\left\{\pi_{\iota\sigma}((t_1, 0)), \pi_{\iota\sigma}\left(\left(\frac{1}{t_1}, 0\right)\right), \pi_{\iota\sigma}((t_2, 0)), \pi_{\iota\sigma}\left(\left(\frac{1}{t_2}, 0\right)\right), \pi_{\iota\sigma}((0, 1)), \pi_{\iota\sigma}(\infty)\right\}.$$ Hence, the defining equation of $C_{\iota\sigma}$ is $$y^2 = x(x - t_1^2)\left(x - \frac{1}{t_1^2}\right)(x - t_2^2)\left(x - \frac{1}{t_2^2}\right) = x(x^2 + ax + 1)(x^2 + bx + 1).$$
Analogously, $$W(C_{\sigma\tau}) = \{\pi_{\sigma\tau}((t_1,0)), \pi_{\sigma\tau}((t_2,0)), \pi_{\sigma\tau}((- t_1,0)),
\pi_{\sigma\tau}((- t_2,0)), \pi_{\sigma\tau}((i, y_0)), \pi_{\sigma\tau}((-i, y_0))\},$$ where $y_0$ is such that $y_0^2 = (2 - a)(2 - b)$. It follows that the defining equation of $C_{\sigma\tau}$ is $$y^2 = (x^2 + 4)\left(x^2 - \left(t_1 - \frac{1}{t_1}\right)^2\right)\left(x^2 - \left(t_2 - \frac{1}{t_2}\right)^2\right) = (x^2 + 4)(x^2 + a + 2)(x^2 + b + 2).$$

One can find equations of remaining elliptic curves using the fact that they are quotients of $C_{\iota\tau}$ and $C_{\iota\sigma\tau}$. For instance, since $E_{\iota\sigma, \iota\tau} = C_{\iota\tau}/\iota\sigma$, the Weierstrass points of $E_{\iota\sigma, \iota\tau}$ are $-a + 2, -b + 2, \infty, 4$. Applying the projective transformation translating these points by $-2$ we get the defining equation  of $E_{\iota\sigma, \iota\tau}$ $$y^2 = (x + a)(x + b)(x + 2).$$ One finds equations of the remaining elliptic curves in a similar way. 

\end{proof}

\begin{rem}
    
    The equation of $\Tilde{C}$ can be found in \cite{10.1145/860854.860904} and the equations of $E_{\sigma}, E_{\tau}$ and $E_{\sigma\tau}$ can be found in \cite{eq}.
\end{rem}

We can describe the correspondence established in Theorem \ref{moduli} in terms of equations. This will help us to describe an involution of the moduli space in a convenient form (see Proposition \ref{inv}).

\begin{prop}
\label{mod}
Let $\Tilde{C}$ be a hyperelliptic genus 3 curve defined by the equation $$y^2 = (x^4 + ax^2 + 1)(x^4 + bx^2 + 1).$$ Then the 5-tuple of points with a distinguished triple on the projective line corresponding to $\Tilde{C}$ is $$([-a:1], [-b:1]; [1:0], [2:1], [-2:1]).$$ 

Conversely, given a 5-tuple of points with a distinguished triple on the projective line, we can find a projective transformation to a 5-tuple of the form $$([-a:1], [-b:1]; [1:0], [2:1], [-2:1])$$ for some $a,b \in \mathbb{C}$. Then the curve $\Tilde{C}$, corresponding to the 5-tuple is given by the equation $$y^2 = (x^4 + ax^2 + 1)(x^4 + bx^2 + 1).$$
\end{prop}

\begin{proof}
    Recall that we have a commutative diagram 
    $$\begin{tikzcd}
\Tilde{C} \arrow[dd, "\pi_{\Tilde{C}}"'] \arrow[rr, "f"] &  & \mathbb{P}^1 \arrow[dd, "g"] \\
                                                           &  &                                \\
\mathbb{P}^1 \arrow[rr, "h"]                             &  & \mathbb{P}^1                  
\end{tikzcd}$$
    One can see from the fact that $\overline{\sigma}([x:y]) = [-x:y]$ and $\overline{\tau}([x:y]) = [y:x]$ that the map $h$ is given by $h([x:y]) = [x^4 + y^4: x^2y^2]$ (here we use notation from Proposition \ref{mainpeq}). Also note that \begin{align*}
        \Fix(\overline{\sigma}) &= \{[1:0], [0:1]\}, \\
        \Fix(\overline{\tau}) &= \{[1:1], [-1:1]\}, \\
        \Fix(\overline{\sigma}) &= \{[i:1], [-i:1]\}. 
    \end{align*}
    Thus, the 5-tuple we get is $$(h([t_1:1]), h([t_2:1]); h([1:0]), h([1:1]), h([i:1]) = ([-a:1], [-b:1]; [1:0], [2:1], [-2:1]).$$
    The second part of the proposition follows from the fact that the 5-tuple with distinguished triple is defined up to projective equivalence.
\end{proof}

In the following proposition we find the involution on the moduli space corresponding to the involution exchanging points in the fibres of the Prym map in Theorem \ref{thmPrym2}.

\begin{prop}
\label{inv}
    Let 
    \begin{multline}
    \label{j}
        \varphi: \mathcal{R}^{H}_{1,4} \ni ([-a:1], [-b:1]; [1:0]; [2:1], [-2:1]) \dashmapsto\\ ([-b:1], [-a:1]; [1:0]; [-2-a-b:1], [-2:1]) \in \mathcal{R}^{H}_{1,4}
    \end{multline} be the rational involution on $\mathcal{R}^{H}_{1,4}$. Here we modify the notation from Proposition \ref{mod} to indicate, that $[1:0]$ is the distinguished point in the triple.
    
    Then $\Pr(T) = \Pr(\varphi(T))$ whenever $\varphi$ is defined on $T \in \mathcal{R}^{H}_{1,4}$. 
\end{prop}

\begin{proof}
    Let $\tilde{C}/E_{\sigma}$ and $\tilde{C}'/E_{\sigma}'$ be the coverings associated to the 5-tuple $$([-a:1], [-b:1]; [1:0]; [2:1], [-2:1]),$$ and $$([-b:1], [-a:1]; [1:0]; [-2-a-b:1], [-2:1])$$ respectively. Note that the latter 5-tuple is projectively equivalent to $$([-a:1], [-b:1]; [1:0]; [2:1], [2 - a - b:1]).$$
    By looking at 2-torsion points we obtain $E_{\iota\sigma,\tau}$ is isomorphic to $E_{\iota\sigma, \tau}'$ and 
    $E_{\iota\sigma, \iota\tau}$ is isomorphic to $E_{\iota\sigma, \iota\tau}'$.
    
    Applying the projective transformation given by $\begin{pmatrix} 4 & 8 + 2a + 2b \\ 0 & - a - b \end{pmatrix}$ to the 5-tuple $([-b:1], [-a:1]; [1:0]; [-2-a-b:1], [-2:1])$ we have another description of $\varphi$
    \begin{multline}
        \varphi([-a:1], [-b:1]; [1:0], [2:1], [-2:1]) =\\ \left(\left[-\frac{2b - 2a - 8}{a + b}:1\right],\left[-\frac{2a - 2b - 8}{a + b}:1\right]; [1:0]; [2:1], [-2:1]\right). 
    \end{multline}
    One can check that $j$-invariants of $E_{\tau}$ (resp. $E_{\sigma\tau}$) and $E_{\tau}'$ (resp. $E_{\sigma\tau}'$) are equal, hence the respective curves are isomorphic. 

    It follows that the two-torsion points $e_1$ and $e_1'$ defining the coverings $E_{\tau} \to E_{\iota\sigma, \tau}$ and $E_{\tau}' \to E_{\iota\sigma, \tau}'$ are the same (similarly for another pair of covers), hence $$P(\Tilde{C}/E_{\sigma}) = E_{\tau} \times E_{\sigma\tau}/\langle e_1 + f_1
    \rangle = E_{\tau}' \times E_{\sigma\tau}'/\langle e_1' + f_1' \rangle = P(\Tilde{C}'/E_{\sigma}')$$ which shows the claim. 
\end{proof}

\begin{rem}
\label{318}
    Note that if $b = -a$ the curves $E_{\tau}$ and $E_{\sigma\tau}$ are isomorphic (see Prop \ref{eqns}), hence on such curves $\Pr$ has degree one (see Remark \ref{ramif}) and $\varphi$ is not defined on such covers. 
    
    Moreover, one can compute that for $b \neq -a$ the 5-tuples $T$ and $\varphi(T)$ are not projectively equivalent, hence top curves are not isomorphic and the degree of the Prym map is precisely two at such points.
\end{rem}

\subsection{Period matrix of $P(\Tilde{C}/E_{\sigma})$}

Let us compute a period matrix of $P(\Tilde{C}/E_{\sigma})$. We know that $P(\Tilde{C}/E_{\sigma}) = (E_{\tau} \times E_{\sigma\tau})/\langle e_1 + f_1 \rangle$. Possibly after a symplectic transformation, we can assume that small period matrices of $E_{\tau}$ and $E_{\sigma\tau}$ are given by $[z_1]$ and $[z_2]$ respectively with $e_1 + f_1$ given by $\begin{bmatrix}
    1 \\ 1
\end{bmatrix}$ (note that the polarization on the product is $(2,2)$).

\begin{prop}
    With the assumptions above, a (big) period matrix of $P(\Tilde{C}/E_{\sigma})$ is $$\begin{bmatrix}
        z_1 & z_1 & 1 & 0 \\
        z_1 & z_1 + z_2 & 0 & 2
    \end{bmatrix}$$ The embeddings of $E_{\tau}$ and $E_{\sigma\tau}$ are given by $x \mapsto \begin{bmatrix}
    x \\ x
\end{bmatrix}$ and $x \mapsto \begin{bmatrix}
    0 \\ x
\end{bmatrix}$ respectively.
\end{prop}

\begin{proof}
   A big period matrix of a $(2,2)$-polarized product $E_{\tau} \times E_{\sigma\tau}$ is $$M = \begin{bmatrix}
        z_1 & 0 & 2 & 0 \\
        0 & z_2 & 0 & 2
    \end{bmatrix},$$ and we denote the column vectors of this matrix as $f_1, f_2, e_1, e_2$. Note that $f_1' = f_1$, $f_2' = f_2 - f_1$, $e_1' = e_1 + e_2$, $e_2' = e_2$ is a symplectic basis with respect to $e^{\mathcal{O}_E(2) \boxtimes \mathcal{O}_F(2)}$. Writing $M$ in the new basis and taking quotient by $e_1'/2$ we get the period matrix of $P(C/E_{\tau})$: $$M' = \begin{bmatrix}
        z_1 & z_1 & 1 & 0 \\
        0 & z_2 & -1 & 2
    \end{bmatrix}.$$ Finally, applying symplectic transformation of the form $\begin{bmatrix}
    1 & 0 \\ 1 & 1
    \end{bmatrix}$ we get $$\begin{bmatrix}
        z_1 & z_1 & 1 & 0 \\
        z_1 & z_1 + z_2 & 0 & 2
    \end{bmatrix}.$$ One easily checks that embeddings of the curves $E_{\tau}$ and $E_{\sigma\tau}$ are given by $x \mapsto \begin{bmatrix}
    x \\ x
\end{bmatrix}$ and $x \mapsto \begin{bmatrix}
    0 \\ x
\end{bmatrix}$. 
\end{proof}

We would like to conclude the paper with the following problems.
\begin{enumerate}
    \item Is it possible to naturally compactify the domain, codomain and the Prym map from Theorem \ref{thmPrym2}?
    \item The investigated moduli space (Theorem \ref{moduli}) is of dimension 2. Therefore, one can expect that if three elliptic curves satisfy some relation there should exists a hyperelliptic genus 3 curve that is a covering of these three elliptic curves. Can one find such an explicit relation?    
    \item We have been able to compute period matrices of Prym varieties. Is it possible to use them to construct explicit period matrices of completely decomposable Jacobians of hyperelliptic genus 3 curves? Note that our construction lives in the boundary of a family considered in \cite{EF}.  
\end{enumerate}

\bibliographystyle{alpha}
\bibliography{sample}

\textsc{P. Bor\'owka, Institute of Mathematics, Jagiellonian University in Krak\'ow, Poland}\\
\textit{email address:} pawel.borowka@uj.edu.pl

\textsc{A. Shatsila, Institute of Mathematics, Jagiellonian University in Krak\'ow, Poland}\\
\textit{email address:} anatoli.shatsila@student.uj.edu.pl.

\end{document}